\documentclass[preprint,11pt]{elsarticle}
\usepackage{amsmath}
\usepackage{amssymb}
\usepackage{latexsym}
\usepackage{graphicx}

\usepackage[usenames]{color}

\input amssym.def
\newsymbol\rtimes 226F
\newfont{\nset}{msbm10}

\newtheorem{theo}{Theorem}[section]
\newtheorem{theorem}[theo]{Theorem}

\newtheorem{lemma}[theo]{Lemma}
\newtheorem{definition}[theo]{Definition}

\newtheorem{coro}[theo]{Corollary}

\journal{Discrete Applied Mathematics}

\begin{document}

\begin{frontmatter}

\title{The Number of Spanning Trees in Apollonian Networks
}

\author{Zhongzhi Zhang, Bin Wu }
\address{School of Computer Science and  Shanghai Key Lab of Intelligent Information Processing,
      Fudan  University, Shanghai 200433, China  ({\tt zhangzz@fudan.edu.cn}).}
\author{Francesc Comellas}
\address{Dep. Matem\`atica Aplicada IV, EETAC,
     Universitat Polit\`ecnica de Catalunya, c/ Esteve Terradas 5, Castelldefels (Barcelona), Catalonia,
     Spain ({\tt comellas@ma4.upc.edu}).}


\begin{abstract}
In this paper we find an exact analytical expression for the number of spanning trees in Apollonian networks. This parameter  can be related to significant topological and dynamic properties of the networks, including percolation, epidemic spreading, synchronization, and random walks. As Apollonian networks constitute an interesting family of maximal planar graphs which are
simultaneously small-world, scale-free, Euclidean and  space filling, modular and highly clustered, the study of their spanning trees is of particular relevance. Our results allow also the calculation of  the spanning tree entropy of  Apollonian networks, which we compare with those of other graphs with the same average degree.
\end{abstract}



\end{frontmatter}

\section{Apollonian networks}
In the process known as Apollonian packing~\cite{KaSu43}, which dates back to Apollonius of Perga (c262--c190 BC),
we start with three mutually tangent circles, and draw their inner Soddy circle (tangent to the three circles).
Next we draw the inner Soddy circles of this circle with each pair of the original three,
and the process is iterated, see Fig.~\ref{Fig.1}.

\begin{figure}[htbp]
 \centering
 \includegraphics[width=0.7\textwidth]{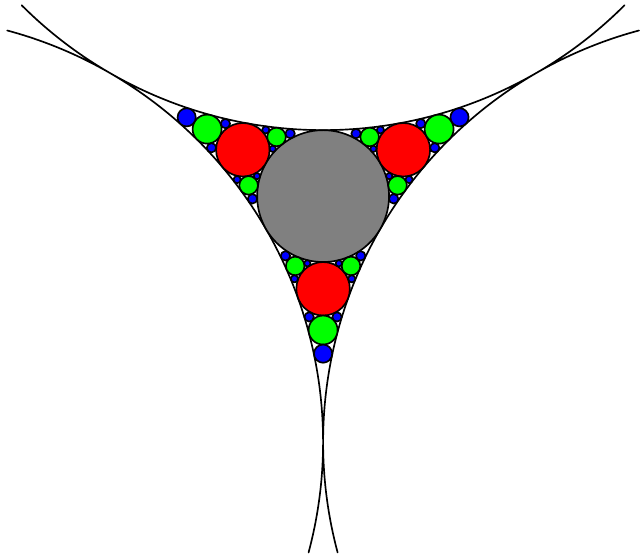}
 \caption{First stages of an Apollonian packing process.}
 \label{Fig.1}
 \end{figure}

An Apollonian packing can be used to design a graph, when each circle is associated to a vertex of the graph
and vertices are connected if their corresponding circles are tangent.
This graph, known as Apollonian graph or  two-dimensional Apollonian network, was introduced
by Andrade et al. \cite{AnHeAnSi05}  and independently proposed by Doye and Massen in \cite{DoMa05}.

We provide here the formal definition and main topological properties of two dimensional Apollonian networks.
We  use standard graph terminology and the words ``network'' and ``graph'' indistinctly.

\begin{definition}
An Apollonian network  $A(n)$, $n\geq 0$,  is a graph  constructed as follows:

For $n=0$, $A(0)$ is the complete graph $K_3$ $($also called a $3$-clique or triangle$)$.

For  $n\geq 1$, $A(n)$ is  obtained from $A(n-1)$:
For each of the existing subgraphs of $A(n-1)$ that is
isomorphic to a $3$-clique and created at step $n-1$,
a new vertex is introduced and connected to all
the vertices of this subgraph. Figure~\ref{Fig.2} shows this
construction process.
\end{definition}

 \begin{figure}[htbp]
 \centering
 \includegraphics[width=0.8\textwidth]{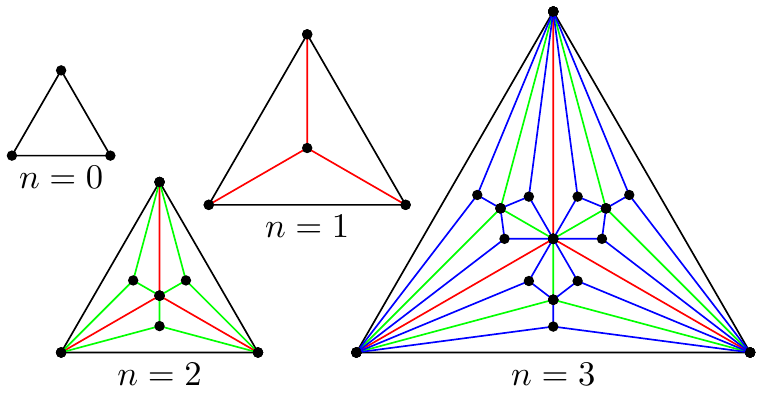}
 \caption{Apollonian graphs $A(n)$ produced at  iterations $n=0,1, 2, 3$ and $4$.}
 \label{Fig.2}
 \end{figure}

The order and size of an Apollonian graph $A(n)=(V(n),E(n))$  are $V_n=|V(n)|=\frac{1}{2}(3^n+5)$ and 
$E_n=|E(n)|=\frac{3}{2}(3^n+1)$.
The graph is scale-free with a power law degree distribution with exponent
$-\ln 3/ \ln 2$.
Many real networks share this property with exponent values in the same range as $A(n )$~\cite{Ne03}.
From the Pearson correlation coefficient for the degrees of the endvertices of the edges of $A(n)$ 
the exact value of the correlation coefficient can be obtained and it is always negative and  goes to zero as the order of the graph increases.
Thus the network is disassortative. 
Most technological and biological networks are disassortative as it is also the case of some information
networks, see~\cite{Ne03,SoVa04}.
It is also possible to obtain the exact analytical value of the  average distance of $A(n)$~\cite{ZhChZhFaGuZo08} which, for $n$ large,  follows  $\bar{d}(n) \sim \ln |V_{n}|$  and shows a logarithmic scaling with the order of the graph.
As the diameter has a similar behavior~\cite{ZhCoFe05},  the graph is small-world.
Moreover, Apollonian graphs are maximally planar, modular, Euclidean and  space filling~\cite{AnHeAnSi05,ZhYaWa05}.
Dynamical processes taking place on these networks, such as percolation, epidemic spreading,
synchronization and random walks, have been also investigated, see~\cite{HuXuWuWa06,ZhRoCo05,ZhRoZh06,ZhYaWa05,ZhZh07}.
Some authors even suggest that the topological and dynamical properties of Apollonian networks are  characteristic of neuronal networks as in  the brain cortex~\cite{PeArHePe07}.

In this paper we study the number of spanning trees of two-dimensional Apollonian networks.
This study is relevant given the importance of the graphs, and because
the number of spanning trees of a finite graph is
a  graph invariant  which
characterizes the reliability of a network \cite{Co87}
and is  related to its  optimal synchronization and the study of random walks \cite{Ma00}.
The number of spanning trees of a graph can be obtained from the product of all nonzero eigenvalues of the
Laplacian matrix of the graph~\cite{GoRo01} (Kirchhoff's matrix-tree theorem).
However, although this result  can be applied to any graph,  this
calculation is analytically and computationally demanding. 
In~\cite{LiWuZhCh11}, the number of spanning trees of two-dimensional Apollonian networks is found without an explicit proof, by using
Kirchhoff's theorem and a recursive evaluation of determinants.
Here, we  follow a different approach. Our method provides the number of spanning trees in Apollonian networks  through a process based on the self-similarity of graphs. 
The main advantage of this method is that it uses a recursive enumeration of subgraphs. Thus, the final tree count does not rely on results published elsewhere and the proof is self-contained.

\section{The number of spanning trees in Apollonian networks}

In this section we find the number of spanning trees of the Apollonian network $A(n)$.
For this calculation we apply a method~\cite{KnVa86} which has been used to find the number of spanning trees in other recursive graph families like the Sierpi\'nski
gasket~\cite{ChChYa07,TeWa06}, the pseudofractal web~\cite{ZhLiWuZo10}, and
some fractal lattices~\cite{Dh77,TeWa11,ZhLiWuZo11}. The main result can be stated as follows.

\begin{theorem}
\label{th:21} The number of spanning trees of the Apollonian network $A(n)$ is
\begin{equation*}
s_n=\frac{1}{4}3^{\frac{3}{4}(-1+3^{n-1}-2(n-1))}5^{\frac{1}{4}(-3+3^n-2(n-1))}(3^n+5^n)^2.
\end{equation*}
\end{theorem}
The definitions and lemmas that follow provide the proof of this theorem.
\medskip

From Fig.~\ref{Fig.2}, we  see that  Apollonian networks are  self-similar, 
suggesting  an alternative way to construct them.
As shown in Fig.~\ref{Fig.3}, $A(n+1)$ can be obtained by joining three replicas of $A(n)$, labeled by $A(n)^1$, $A(n)^2$ and $A(n)^3$, and merging three pairs of edges. 
This particular structure of  Apollonian networks allow us to write recursive equations
for the number of spanning trees, which are solved by induction.

 \begin{figure}[htbp]
 \centering
 \includegraphics[width=0.8\textwidth]{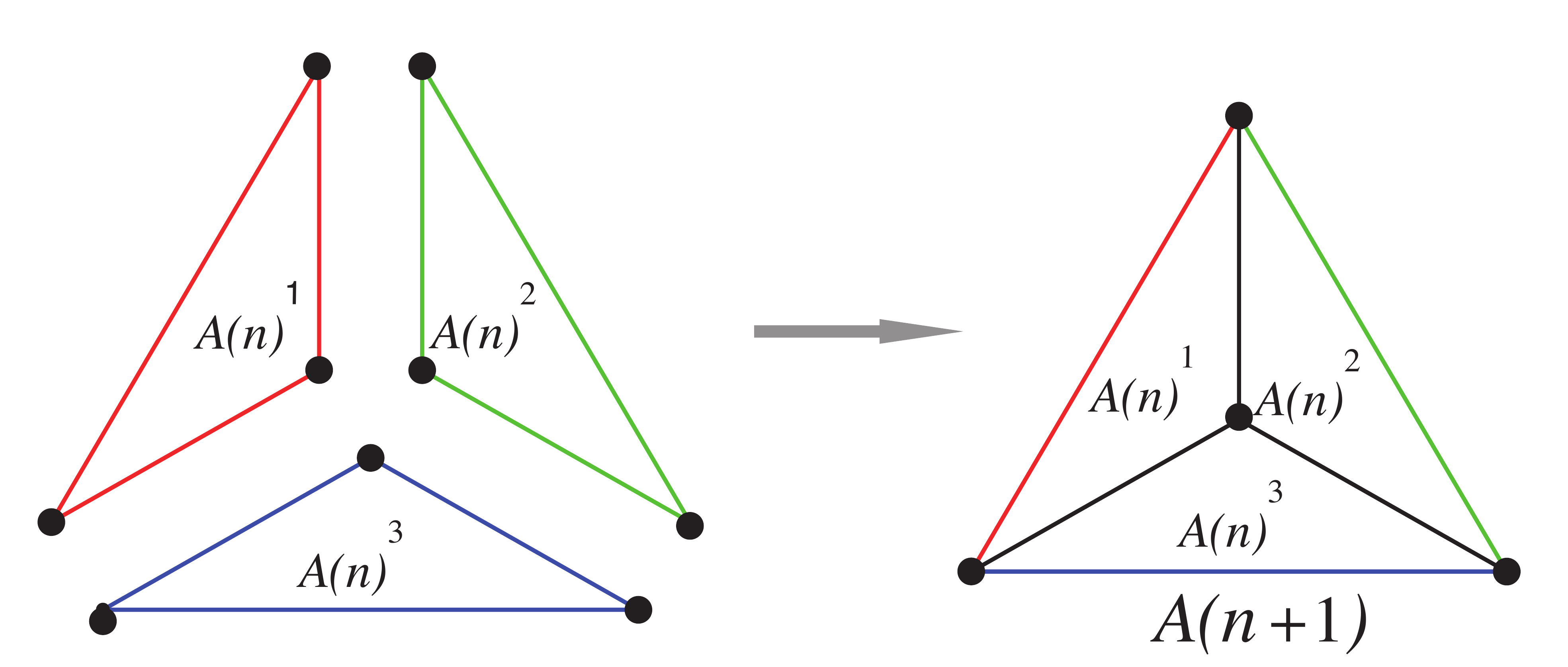}
 \caption{Recursive construction of Apollonian networks, pointing out their self-similarity.
$A(n+1)$ can be obtained by joining three replicas of $A(n)$, labeled   here  $A(n)^1$, $A(n)^2$ and $A(n)^3$,
after merging three pairs of edges.}
 \label{Fig.3}
 \end{figure}

In the following, we denote by $V_n$ and $E_n$ the number of vertices and edges of
$A(n)$. A spanning subgraph of $A(n)$ is a subgraph with the same
vertex set as $A(n)$ and a number of edges $E_n^\prime$ such that
$E_n^\prime \le E_n$. A spanning tree of $A(n)$ is a spanning subgraph
which is a tree and thus $E_n^\prime = V_n-1$.

We call ``hub vertices''  the three outmost vertices in the construction as shown in Fig.~\ref{Fig.2}
and ``hub edges'' the three exterior edges which connect the hub vertices.

To simplify our calculations, we introduce the following five classes of spanning subgraphs of  $A(n)$, see Fig.~\ref{Fig.4}:
Class $A_n$ has all spanning subgraphs of $A(n)$ which consist of three trees and such that each 
hub vertex of $A(n)$ belongs to a different tree.
Next three classes contain those spanning subgraphs of $A(n)$ which consist
of two trees such that no hub edges belong to the spanning subgraph and one of  the hub vertices
of the subgraph belongs to one tree and the other two hub vertices are in the second tree.
By taking into account the tree to which a given hub vertex belongs we have classes $B_n$, $B^{\prime}_n$ and  $B^{\prime\prime}_n$, 
Note that  all subgraphs in each of these classes can be obtained, by a given symmetry, from those in any of the other two classes (see ~Fig.~\ref{Fig.4}).
Finally, class $C_n$  contains all spanning trees  of $A(n)$ which have no hub edges.
 \begin{figure}[htbp]
 \centering
 \includegraphics[width=0.9\textwidth]{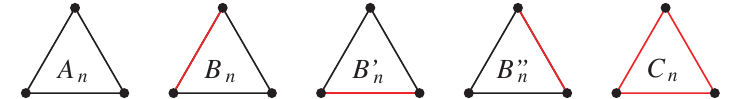}
 \caption{Schematic illustration of the five classes  of spanning subgraphs  $A_n$, $B_n$, $B^{\prime}_n$, $B^{\prime\prime}_n$  and $C_n$ derived from $A(n)$.
 In this figure, vertices represent the hub vertices of $A(n)$ and two hub vertices
 joined by a black line belong to different trees. A red line means that the two hub vertices are in the same tree but have no hub edge joining them.}
 \label{Fig.4}
 \end{figure}

These classes have cardinality $a_n$, $b_n$, $b^{\prime}_n$, $b^{\prime\prime}_n$ and $c_n$, respectively. Note that $b_n=b^{\prime}_n=b^{\prime\prime}_n$.
We denote as  $s_n$ the total number of spanning trees of $A(n)$.
In Fig.~\ref{Fig.5} we show the elements of these classes for $n=0, 1$.

This classification is introduced to facilitate the iterative calculation of the number of spanning trees 
as all spanning trees of  $A(n+1)$ can be constructed  from  subgraphs of $A(n)$ through the merging process introduced above (Fig.~\ref{Fig.3}).

 \begin{figure}[htbp]
 \begin{center}
 \includegraphics[width=0.9\textwidth]{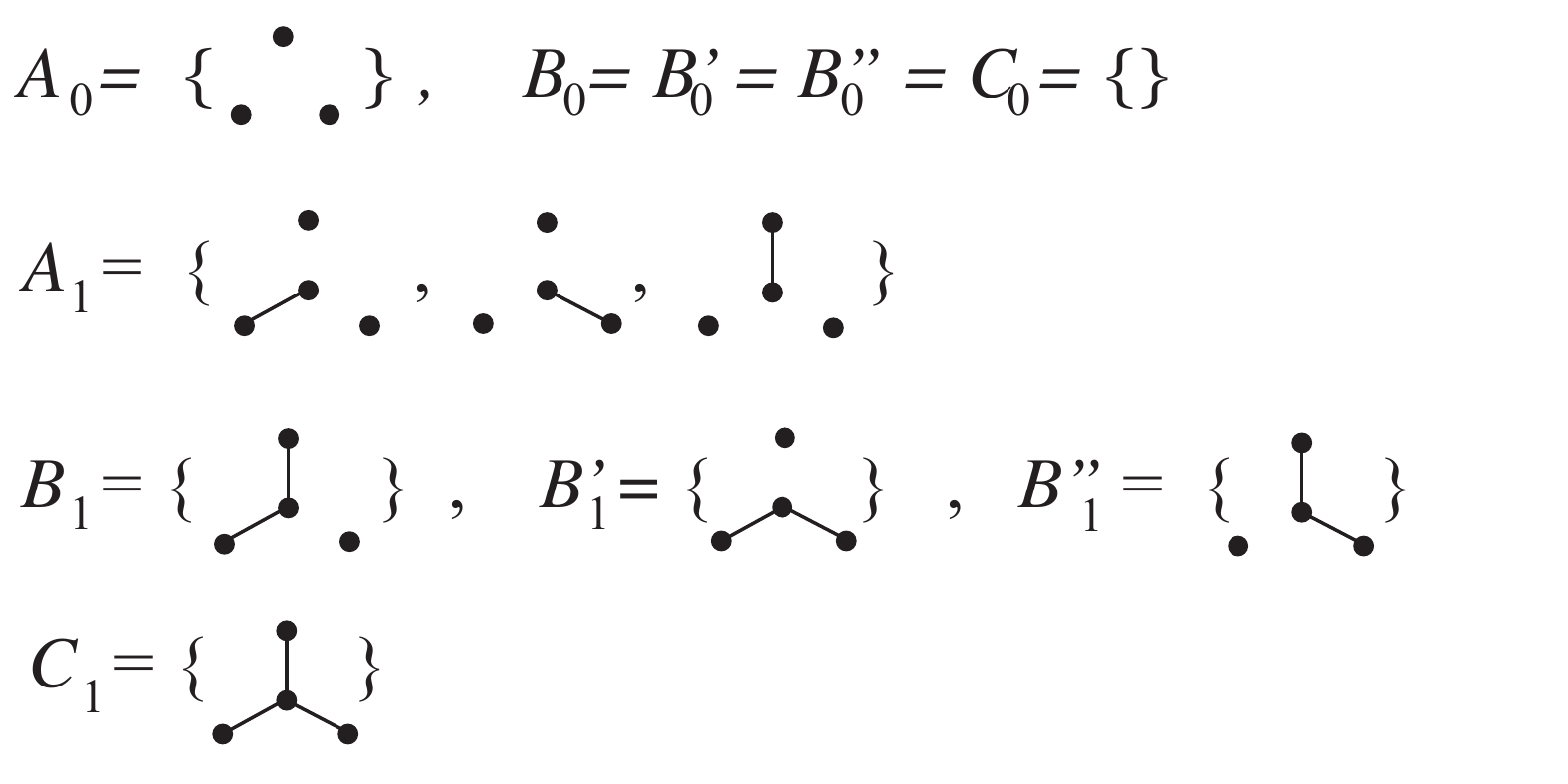}
 \caption{Subgraph classes $A_n$, $B_n$, $B^{\prime}_n$, $B^{\prime\prime}_n$  and $C_n$ for $n=0, 1$.
 Thus $a_0=1$, $b_0=b^{\prime}_0=b^{\prime\prime}_0=c_0=0$ and $a_1=3$, $b_1=b^{\prime}_1=b^{\prime\prime}_1=c_1=1$.}
 \label{Fig.5}
 \end{center}
 \end{figure}

In the previous definitions, we have not considered the cases where  the spanning subgraph contains hub edges.
We deal with these cases in the following lemma.

\begin{lemma}\label{lmm:22}
\null 
\begin{enumerate}
\renewcommand{\labelenumi}{\alph{enumi})}
\item The number of spanning subgraphs of $A(n)$ which consist of two trees such that one hub edge
with its two hub vertices belongs to one tree while  the third hub vertex of  $A(n)$ is in the other tree equals $a_n$.

\item The number of spanning subgraphs of $A(n)$ such that they contain just one hub edge and one  hub vertex
which is connected to one of the hub vertices of this edge through edges of the tree is $b_n$.

\item The number of spanning subgraphs of $A(n)$ that include two hub edges is $a_n$.
\end{enumerate}
\end{lemma}
 \begin{figure}[htbp]
 \centering
 \includegraphics[width=0.9\textwidth]{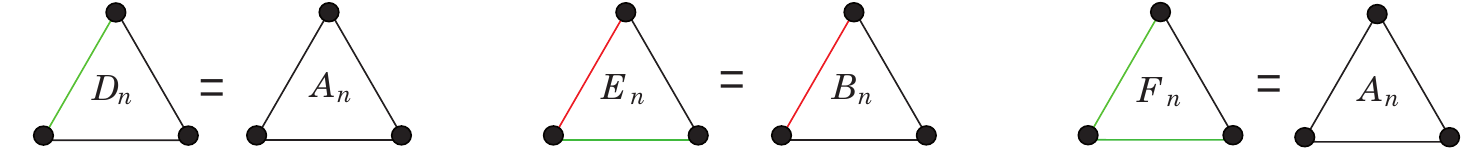}
 \caption{Graphical illustrations for Lemma~\ref{lmm:22}.  The green line means that the hub edge joining
 these two hub vertices  belongs to the spanning subgraph.}
 \label{Fig.6}
 \end{figure}

\begin{proof}
\begin{enumerate}
\renewcommand{\labelenumi}{\alph{enumi})}
\item Let  $D_n$ be the set of subgraphs considered in (a).
We verify the correctness of the result by showing that there exists an one-to-one correspondence between
the set $D_n$ and class $A_n$.  
For every spanning subgraph in $D_n$, if we remove the hub edge, then the three hub vertices
will  belong to three different trees, so it belongs to $A_n$, see Fig.~\ref{Fig.6}.
Conversely, for every spanning subgraph in $A_n$, if we add a hub edge, then its two
hub vertices belong to one tree and the subgraph is in $D_n$.
Thus, there exists a one-to-one correspondence between $D_n$ and $A_n$,  and the cardinality
of $D_n$ is $a_n$.

\item Let $E_n$ be the set of subgraphs considered in (b).  As above, we can verify an one-to-one correspondence
between sets  $E_n$  and $B_n$  by deleting the hub edge.
Thus, the cardinality of $E_n$ is $b_n$.

\item Consider the bijection between $F_n$, the set of subgraphs which contain two hub edges, and $A_n$ (Fig.~\ref{Fig.6}).
\end{enumerate}

\end{proof}

\bigskip


The following four lemmas establish recursive relationships among the  parameters
$a_n$, $b_n$, $c_n$ and $s_n$.

\begin{lemma}
\label{lmm:23}
For $n \geq 0$, $a_{n+1}=3a_n^3+6a_n^2b_n$.
\end{lemma}
\begin{proof}
We prove this result by considering a graphical version of the equation (Fig.~\ref{Fig.8}) which represents the recursive construction method of $A(n+1)$ from $A(n)$ and enumerates all possible contributions to $a_{n+1}$.

In this representation we only draw four vertices in each case, since each non drawn (interior) vertex connects at least  to one of these four vertices (although they do not have necessarily to be adjacent). This is sufficient to determine whether each case belongs to $A_{n+1}$, $B_{n+1}$ or $C_{n+1}$ .
\begin{figure}[htbp]
 \centering
  \includegraphics[width=0.5\textwidth]{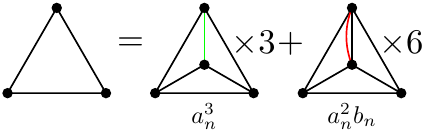}
 \caption{Configurations needed to find $a_{n+1}$ from $a_{n}$ and $b_n$.
 In this representation,  the red curve denotes  the spanning tree to which the two hub vertices belong to.
 The number on the right of the figure counts  configurations that, by symmetry, contribute
 to $a_{n+1}$  through the merging process.  The figure also identifies different contributions form  merging subgraphs.} \label{Fig.8}
 \end{figure}

Next we should prove that each configuration is correct, but we only analyze in detail the first additive term as the other term can be verified in a similar way.
For this case (see Fig.~\ref{Fig.100}), hub vertices $h_1$ and $h_4$, according to the merging process described at the beginning of this Section, belong to subgraphs in  both copies $A(n)^1$  and $A(n)^2$ where they are adjacent while $h_2$ and $h_3$ are in different copies. Thus, after merging these two edges, there are $a_n^3$ subgraphs which belong to $A_{n+1}$. Because of the symmetry, $h_4$ could also be adjacent to $h_2$ or $h_3$ (instead of being adjacent $h_1$) and we count three times this case.
\end{proof}

\ \begin{figure}[htbp]
 \centering
\includegraphics[width=0.4\textwidth]{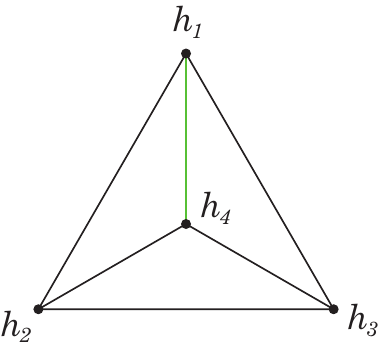}
 \caption{The first configuration which contributes to $a_{n+1}$.}
 \label{Fig.100}
 \end{figure}

\begin{lemma} \label{lmm:24}
For $n \geq 0$, $b_{n+1}=a_n^3+7a_n^2b_n+7a_nb_n^2+a_n^2c_n$.
\end{lemma}
\ \begin{figure}[htbp]
 \centering
 \includegraphics[width=\textwidth]{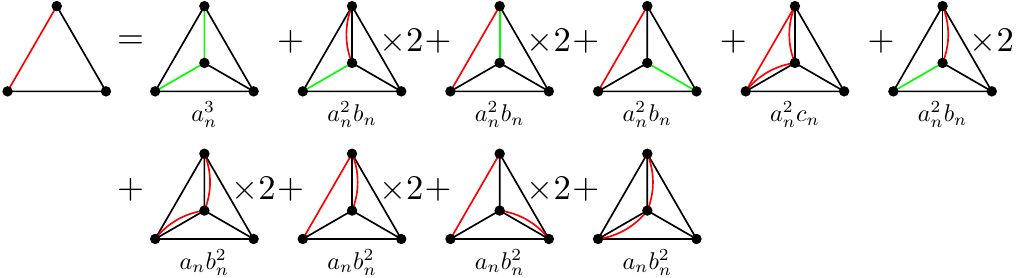}
 \caption{Configurations needed to find $b_{n+1}$.}
 \label{Fig.9}
 \end{figure}
\begin{proof}
We prove the lemma by enumeration. Figure~\ref{Fig.9} shows all the distinct possibilities.
Again, we only analyze the first case. We label the four hub vertices in the same way as in Fig.~\ref{Fig.100}.
In the first case, $h_1$, $h_2$, $h_4$ are all connected while $h_3$ is not. There are two spanning trees, and one has no hub edges, so this configuration belongs to set $B_{n+1}$.
Symmetries generate equivalent configurations and the factor is one.
\end{proof}

\begin{lemma} \label{lmm:25}
For $n \geq 0$, $c_{n+1}=a_n^3+12a_n^2b_n+36a_nb_n^2+14b_n^3+3a_n^2c_n+12a_nb_nc_n$.
\end{lemma}
\ \begin{figure}[htbp]
 \centering
 \includegraphics[width=\textwidth]{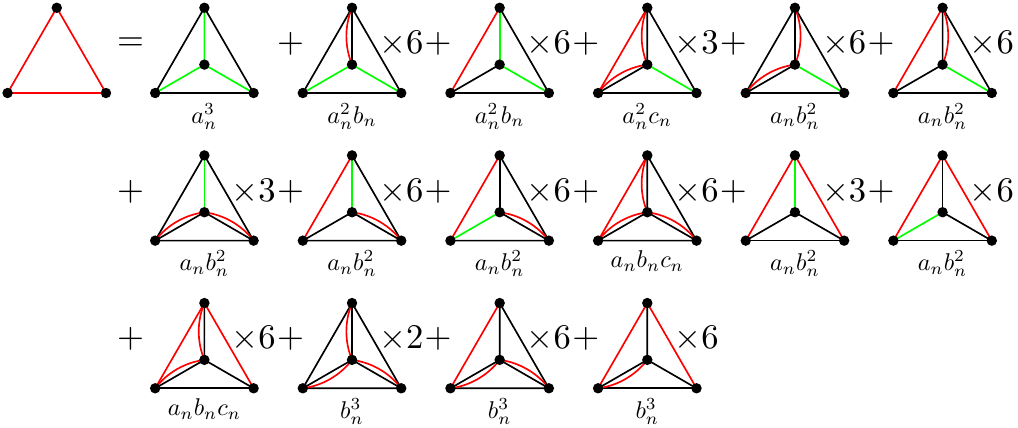}
 \caption{Configurations that contribute to  $c_{n+1}$.}
 \label{Fig.10}
 \end{figure}
\begin{proof}
As in former lemmas, the proof is by enumeration of all possible contributions to $c_{n+1}$, see in Fig.~\ref{Fig.10} the details.
In the first case, $h_1$, $h_2$, $h_3$ and $h_4$ are all connected and the merging process produces a spanning tree.
As no hub edges are included  in it, we can see that this case belongs to set $C_{n+1}$.
Besides, because of the symmetry, only this configuration is relevant.
All other cases are analyzed similarly and  we omit the details.
\end{proof}

\begin{lemma} \label{lmm:26}
For $n \geq 0$, $s_{n+1}=16a_n^3+72a_n^2b_n+78a_nb_n^2+14b_n^3+9a_n^2c_n+12a_nb_nc_n$.
\end{lemma}
\begin{figure}[htbp]
 \centering
 \includegraphics[width=\textwidth]{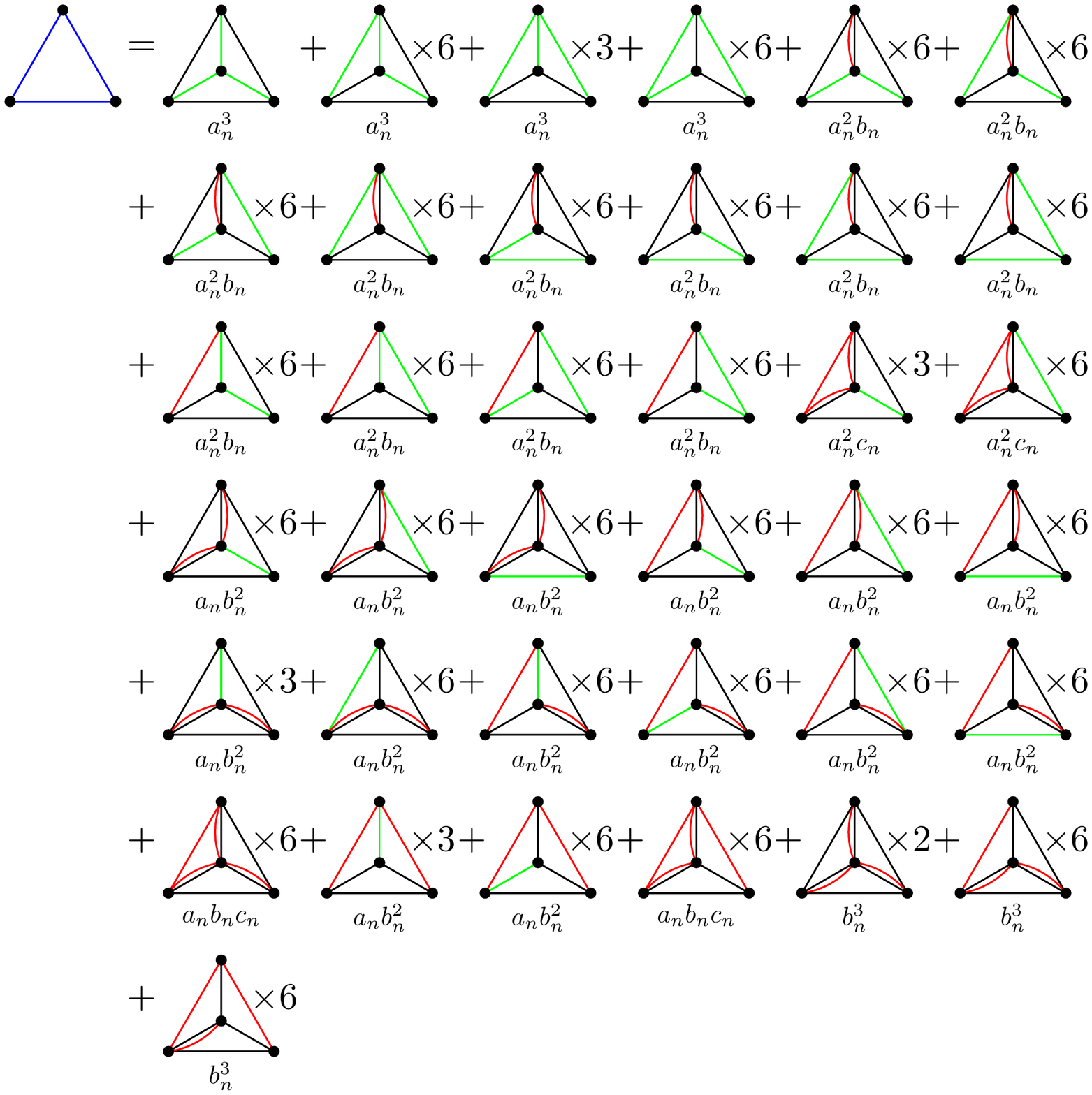}
 \caption{Illustration for the calculation of $s_{n+1}$. Here, a blue line indicates that the corresponding  hub vertices  belong to the same spanning tree (including both the cases where they are adjacent and  they are connected through other vertices).}
 \label{Fig.11}
 \end{figure}
\begin{proof}
Figure~\ref{Fig.11} shows all the configurations contributing $s_{n+1}$.
We do not give the calculation details as they are like in previous lemmas.
\end{proof}
\bigskip


The next three lemmas give the values of $a_n$, $b_n$ and $c_n$.

\begin{lemma} \label{lmm:27}
For $n \geq 0$, $a_n=3^{-\frac{1}{4}+\frac{3^n}{4}+\frac{n}{2}}5^{-\frac{1}{4}+\frac{3^n}{4}-\frac{n}{2}}$.
\end{lemma}
\begin{proof}
To obtain this closed-form expression, Lemma~\ref{lmm:23} will provide a recursive 
equation for  $a_n$.
Thus,  we use Lemma~\ref{lmm:24} to write $b_n$ in terms of $a_n$  and,  as a preliminary result, 
we need to prove that $a_n c_n=3 b_n^2$:

We use induction. For $n=0$,  the initial conditions $a_0=1$ and $b_0=c_0=0$ make the equation true.
Let us  assume that for $n=k$, the equality holds.
For $n=k+1$, and from Lemmas~\ref{lmm:23}-\ref{lmm:25}, we have that
\begin{eqnarray*}
a_{k+1}c_{k+1}-3b_{k+1}^2&=&(3a_k^3+6a_k^2b_k)(a_k^3+12a_k^2b_k+36a_kb_k^2+14b_k^3+3a_k^2c_k\\
&&+12a_kb_kc_k)-3(a_k^3+7a_k^2b_k+7a_kb_k^2+a_k^2c_k)^2\\
&=&3a_k^2(a_k^2+4a_kb_k+7b_k^2-a_kc_k)(a_kc_k-3b_k^2)\,,
\end{eqnarray*}
and as $a_{k}c_{k}-3b_{k}^2=0$ (induction hypothesis), we reach the result.

With this result, we can replace in Lemma~\ref{lmm:24}    $a_n^2 c_n$ by $3 a_n b_n^2$ and as
$\frac{b_n}{a_n}=\frac{a_{n+1}}{6 a_n^3}-\frac{1}{2}$ (Lemma~\ref{lmm:23}), we obtain
\begin{equation*}
\frac{b_{n+1}}{a_n^3}=1+7\frac{b_n}{a_n}+10\left(\frac{b_n}{a_n}\right)^2 =-\frac{a_{n+1}}{2a_n^3}+\frac{5a_{n+1}^2}{18a_n^6},
\end{equation*}
which we can write as
\begin{equation*}
\frac{b_{n+1}}{a_{n+1}}=-\frac{1}{2}+\frac{5a_{n+1}}{18a_n^3}.
\end{equation*}
Fom Lemma~\ref{lmm:23}, we have
\begin{equation*}
\frac{a_{n+2}}{a_{n+1}^3}=3+6\frac{b_{n+1}}{a_{n+1}}=3+6\left(-\frac{1}{2}+\frac{5a_{n+1}}{18a_n^3}\right)=\frac{5a_{n+1}}{3a_n^3}.
\end{equation*}
and, if we define $w_n=\frac{a_{n+1}}{a_n^3}$, we obtain the recursion $w_{n+1}=\frac{5}{3}w_n$ which together with the initial
condition $w_0=\frac{a_1}{a_0^3}=3$  leads to $w_n=\frac{5^n}{3^{n-1}}$ and allows us to write 
\begin{equation*}
a_{n+1}=\frac{5^n}{3^{n-1}}a_n^3.
\end{equation*}
This equation, with the condition  $a_0=1$,  gives
\begin{equation*}
a_n=3^{-\frac{1}{4}+\frac{3^n}{4}+\frac{n}{2}}5^{-\frac{1}{4}+\frac{3^n}{4}-\frac{n}{2}}.
\end{equation*}
\end{proof}

\begin{lemma} \label{lmm:28}
For $n \geq 0$, $b_n=\frac{1}{2}15^{\frac{1}{4}(-1+3^n-2n)}(5^n-3^n)$.
\end{lemma}
\begin{proof}
From Lemma~\ref{lmm:23} we have  $b_n=\frac{a_{n+1}-3a_n^3}{6a_n^2}$  and the result follows
from the expression of $a_n$ found in Lemma~\ref{lmm:27}.
\end{proof}

\begin{lemma} \label{lmm:29}
For $n \geq 0$, $c_n=\frac{1}{4} 3^{\frac{1}{4} \left(3+3^n-6 n\right)} 5^{\frac{1}{4} \left(-1+3^n-2 n\right)} \left(3^n-5^n\right)^2$.
\end{lemma}
\begin{proof}
From the proof of  Lemma~\ref{lmm:27} we have $c_n=\frac{3b_n^2}{a_n}$,
and using  $a_n$ and $b_n$ as found in  former lemmas, we obtain  $c_n$.
\end{proof}
\medskip


The main result of this section, the number of spanning trees of the Apollonian network $A(n)$ (Theorem~\ref{th:21}), follows 
from Lemma~\ref{lmm:26} and the expressions  for $a_n$, $b_n$ and $c_n$ obtained in
Lemmas~\ref{lmm:27}, \ref{lmm:28} and \ref{lmm:29}:
\begin{equation*}
s_n=\frac{1}{4}3^{\frac{3}{4}(-1+3^{n-1}-2(n-1))}5^{\frac{1}{4}(-3+3^n-2(n-1))}(3^n+5^n)^2.
\end{equation*}

\medskip

\section{Spanning tree entropy of  Apollonian networks}

After having an explicit expression for the number of spanning trees
of $A(n)$, we can calculate its spanning tree  entropy, which
is defined as in~\cite{Ly05,Wu77}:
\begin{equation*}
z=\lim_{n \to \infty}\frac{\ln s_n}{V_n}.
\end{equation*}\label{eq:en}



\begin{coro}
The spanning tree entropy of Apollonian networks is $\displaystyle{\frac{\ln 15}{2}}$.
\end{coro}
\begin{proof}
Define $z_n=\frac{\ln s_n}{V_n}$.  From Theorem~\ref{th:21}, we have
\begin{equation*}
z_n=\frac{-8\ln2+\ln\frac{27}{5}+3^n\ln15-2n\ln135+8\ln(3^n+5^n)}{2(5+3^n)}\,,
\end{equation*}
and thus
\begin{equation*}
z=\lim\limits_{n\to \infty} z_n=\frac{\ln 15}{2}.
\end{equation*}
\end{proof}

We can compare this asymptotic value of the entropy  of  the spanning trees  for Apollonian networks,
 $z=\frac{\ln 15}{2}\simeq 1.3540$, with that of other relevant graphs with the same average degree.
For example, the value for the 3-dimensional Sierpinski graph is 1.5694~\cite{ChChYa07}   and
for the 3-dimensional hypercubic lattice  ${\cal{L}}_3$
 is $1.6734$~\cite{FeLy03,ShWu00}.
Thus, the asymptotic value for Apollonian networks reflects the fact that the number of 
spanning trees in $A(n)$, although growing exponentially, do it at a lower rate 
than these graphs which have the same average degree.

This result would suggest that Apollonian networks, as they have fewer spanning trees, 
are less reliable to a random removal of edges than the graphs cited above.  
However, Apollonian networks are scale-free and it is known that graphs with this degree distribution are more resilient than homogeneous graphs, see for example \cite{Tu00}. 
Thus, the particular degree distribution of graphs $A(n)$ might increase their robustness in relation to
regular graphs with the same order and size.  
These considerations indicate that it would be of interest to study the connections among the spanning tree entropy of a graph and other relevant graph parameters like, for example,  degree distribution and degree correlation.

\section{Conclusion}
In this paper we find the number of spanning trees
in Apollonian networks by using a method, based on its self-similar structure,
which allows us to obtain an  exact analytical expression for any number of discs.
The method could be used to further study in this graph, and other self-similar graphs,
their spanning forests, connected spanning subgraphs, random walks
and vertex or edges coverings.
Knowing the  number of spanning trees for Apollonian networks allows us to show
that their spanning tree entropy is lower than in other graphs with the same average degree.

\section*{Acknowledgements}
Z. Zhang is supported by the National Natural Science Foundation of China under Grants No. 61074119. F. Comellas is supported by the Ministerio de Economia y Competitividad, Spain, and the European Regional Development Fund under project MTM2011-28800-C02-01 and partially supported by the Catalan Research Council under grant 2009SGR1387.



\end{document}